\documentclass[reqno]{amsart}
\usepackage[english]{babel}

\usepackage{amsfonts}
\usepackage{amssymb}
\usepackage{color}

\newtheorem{thm}{Theorem}[section]
\newtheorem{cor}[thm]{Corollary}
\newtheorem{lem}[thm]{Lemma}

\theoremstyle{definition}

\newtheorem{exmp}{Example}[section]

\theoremstyle{remark}

\numberwithin{equation}{section}

\begin{document}

\title[The Fredholm  
Navier-Stokes type equations]{The Fredholm 
Navier-Stokes type equations 
for the de Rham complex over weighted H\"older spaces 
}

\author{K.V. Gagelgans}

\address[Kseniya Gagelgans]
{Siberian Federal University
                                                 \\
         pr. Svobodnyi 79
                                                 \\
         660041 Krasnoyarsk
                                                 \\
         Russia}
\email{ksenija.sidorova2017@yandex.ru}

\author{A.A. Shlapunov}

\address[Alexander Shlapunov]
{Siberian Federal University
                                                 \\
         pr. Svobodnyi 79
                                                 \\
         660041 Krasnoyarsk
                                                 \\
         Russia}
\email{ashlapunov@sfu-kras.ru}


\begin{abstract}
We consider a family of initial problems for the Navier-Stokes type equations 
generated by the de Rham complex in ${\mathbb R}^n \times [0,T]$, $n\geq 2$, 
with a positive  time $T$ over a scale weighted anisotropic H\"older spaces. 
As the weights control the order of zero  
at the infinity with respect to the space variables for vectors fields 
under the consideration,  
this actually leads to initial problems over a compact 
manifold  with the singular conic point at the infinity. 
We prove that each  problem from the family induces Fredholm open
injective  mappings on elements of the scales. At the step $1$ of the  
complex we may apply the results to the classical 
Navier-Stokes equations for incompressible viscous fluid. 
\end{abstract}

\maketitle

\section*{Introduction}
\label{s.Int}

The Navier-Stokes equations  describe the dynamics of incompressible viscous fluid
that is of great importance in applications, see, for instance, 
\cite{Tema79}, \cite{Lady70}. Essential contributions has been published in the research articles
   \cite{Lera34a,Lera34b},
   \cite{Kolm42},
   \cite{Hopf51},
as well as surveys and books
   \cite{Lady70},
  \cite{Lion69} 
   \cite{Tema79},
etc. Actually, the problem is solved 
in the frame of the concept of weak solutions, see, J.~Leray \cite{Lera34a,Lera34b}, 
E.~Hopf \cite{Hopf51}, O.A.~Ladyzhenskaya \cite{Lady70}, but 
no general uniqueness theorem for weak 
solutions has been known except the two-dimensional case. As far as we know, 
there are no general results on the global solvability in time for the problem in spaces of 
sufficiently regular vector fields where the uniqueness theorems for it are available, too. 
We point out an important direction related to the problem of the existence of regular 
solutions to the Navier-Stokes equation: S. Smale \cite{Sm65} developed 
the concept of Fredholm non-linear mappings of Banach spaces applicable 
to a wide class of non-linear equations of Mathematical Physics (cf. \cite{Tema95} for 
the steady version of the Navier-Stokes equations). 

Recently, the Navier-Stokes type equations were considered in the frame of 
elliptic differential complexes, see \cite{MShlT19}, \cite{ShlTaArxiv}, 
\cite{ShlTa21}, \cite{Po22} over 
scales of Bochner-Sobolev type spaces parametrized by smoothness index  
$s \in {\mathbb Z}_+$ where the Sobolev embedding 
theorems provide point-wise smoothness for sufficiently large $s$. 

On the other hand, results of paper \cite{PlSv03} demonstrate 
that considering the Navier-Stokes type equations over the whole space 
${\mathbb R}^n \times [0,+\infty)$ it is important 
to control the order of zero 
at the infinity with respect to the space variables for the corresponding solutions. 
Namely, \cite{PlSv03}  provides an instructive  example of a non-linear problem in 
${\mathbb R}^n \times [0,T)$, structurally similar to the Cauchy problem for the 
Navier-Stokes equations and `having the same energy estimate'{}, but,  according to some 
considerations including  numerical simulations, admitting singular 
solutions of special type for smooth data if $n\geq 5$. An essential role 
in the arguments of this paper plays the fact that 
certain asymptotic behaviour of the initial data at the infinity 
with respect to the space variables  prevents blow-up behaviour in a finite time interval 
for the considered solutions, cf. also comments 
by \cite[formulas (4), (5)]{Feff00} related to the data 
of the Navier-Stokes equations for incompressible fluid. 

One of the possibilities deal with the asymptotic was indicated in 
\cite{ShlTa18} where the Navier-Stokes  equations 
for incompressible viscous fluid were considered
in ${\mathbb R}^n \times [0,T]$, $n\geq 3$,  
for a positive  time $T$ over a scale weighted anisotropic H\"older spaces 
with the weights controlling the order of decreasing  
at the infinity with respect to the space variables for the vectors fields 
under the consideration. This actually leads to an initial problem where the space 
variables belong to a compact 
manifold  with the singular conic point at the infinity, cf. \cite{Be11}. 

In the present paper we extend the results of \cite{ShlTa18} to 
a family of initial problems for the Navier-Stokes type equations 
generated by the de Rham complex in ${\mathbb R}^n \times [0,T]$, $n\geq 2$, 
with a positive  time $T$ over a scale of weighted anisotropic H\"older spaces. 
It is worth to say that the problem, discussed in \cite{PlSv03}, is included 
to the consideration. 
Using the recent developments of the Hodge theory for the de Rham complex  over these spaces, 
see \cite{GaSh19}, \cite{Ga22}, we involve  weight indexes $\delta>n/2 $, 
that corresponds  to the asymptotic 
$|x|^{-\delta-|\alpha|}$, $x\in {\mathbb R}^n$, as $|x| \to + \infty$, 
for the related solutions and their partial derivatives of order $\alpha \in {\mathbb Z}_+$.
Namely, we consider the Navier-Stokes type equations in the framework of 
the theory of operator equations in Banach space and 
we prove that each initial problem from the family induces Fredholm open
injective  mappings on elements of the scales. At the step $1$ of the  
complex we may apply the results to the classical 
Navier-Stokes equations for incompressible viscous fluid.  

We do not discuss existence theorems here but we hope that the use 
of the weighted H\"older spaces with proper weight indexes 
may exclude the blow-up behaviour of solutions 
to the Navier-Stokes type equations considered in \cite{PlSv03}.

\section{Function spaces, embedding theorems and a non-linear problem}
\label{s.1}

Let ${\mathbb R}^n$ be the $n$-dimensional Euclidean space with the coordinates  
$x=(x_1, \dots , x_n)$.  To introduce weighted H\"older spaces over  ${\mathbb R}^n$   we set
$$
   w (x)=
 \sqrt{1 + |x|^2}, \,\, 
   w (x,y)
  =
  \max \{ w (x), w (y) \} \sim \sqrt{1 + |x|^2 + |y|^2}
$$
for $x, y \in \mathbb{R}^n$. Let $\delta \in \mathbb{R}$. 
(Note that $\delta$ is tacitly assumed to be nonnegative.) 
For $s = 0, 1, \ldots$, denote by $C^{s,0,\delta} $ the space of all $s$ times continuously differentiable functions on $\mathbb{R}^n$ with finite norm 
$$
   \| u \|_{C^{s,0,\delta} }
 = \sum_{|\alpha| \leq s}
   \sup_{x \in \mathbb{R}^n}
   (w (x))^{\delta+|\alpha|}
   |\partial^\alpha u (x)|.
$$
For $0 < \lambda \leq 1$, we introduce
$$
   \langle u \rangle_{\lambda,\delta}
 = \sup_{{{x,y \in \mathbb{R}^n \atop x \neq y} \atop |x-y| \leq |x|/2}}
   (w (x,y))^{\delta+\lambda} \frac{|u (x) - u (y)|}{|x-y|^\lambda}.
$$
and  we define 
$C^{0,\lambda,\delta} $ to
consist of all continuous functions on $\mathbb{R}^n$ with finite norm
$$
   \| u \|_{C^{0,\lambda,\delta} }
 = \| u \|_{C^{0,\lambda} (\overline{U})}
 + \| u \|_{C^{0,0,\delta} } + 
\langle u \rangle_{\lambda,\delta},
$$
where $U$ is a small neighbourhood of the origin in $\mathbb{R}^n$ and $C^{0,\lambda} (\overline{U})$ 
is the standard H\"older space over the compact $\overline U$.
Finally, for $s \in \mathbb{Z}_{\geq 0}$, we introduce $C^{s,\lambda,\delta} $ to be
the space of all $s$ times continuously differentiable functions on $\mathbb{R}^n$  with finite norm
$$
   \| u \|_{C^{s,\lambda,\delta} }
 = \sum_{|\alpha| \leq s}
   \| \partial^\alpha u \|_{C^{0,\lambda,\delta+|\alpha|} }.
$$

The normed spaces $C^{s,\lambda,\delta}$ constitute a scale of Banach spaces
parametrised by
   $s \in \mathbb{Z}_{\geq 0}$,
   $\lambda \in [0,1]$ and
   $\delta \in \mathbb{R}$.
The properties of the scale (e.g. natural coninuous and compact 
embeddings) are well known, see, for instance, \cite{ShlTa18}, \cite{GaSh19}. 

Next, denote by $\varLambda^q$ the bundle of exterior forms of 
degree $0 \leq q \leq n$ over $\mathbb{R}^n$. 
We write $C^\infty_{\varLambda^q} (\mathbb{R}^n)$ for the space of
all differential forms of degree $q$ with $C^\infty$ coefficients on $\mathbb{R}^n$.
These space constitute the so-called de Rham complex $C^\infty_{\varLambda^\cdot} (\mathbb{R}^n)$ on 
$\mathbb{R}^n$ whose differential is given by the exterior derivative $d$. 
To display $d$ acting on $q\,$-forms one uses the designation  $du := d_q u$ 
for $u \in C^\infty_{\varLambda^q} (\mathbb{R}^n)$ (see for instance \cite{deRh55}); 
it is convenient to set $d_q  =  0$ if $ q<0 $ or $q\geq n$.  
As usual, denote by $d^*_q$ the formal adjoint for $d_q$. 
Then, as it is known,  we have  
\begin{equation}
\label{eq.deRham}
   d_{q+1} \circ d_q 
  =
  0,  d^*_q d_q +   d_{q-1} d^*_{q-1} =
 -  E_{m (q)} \varDelta , \, 0\leq q\leq n,
\end{equation} 
where $E_m$ is the unit matrix of type $(m \times m)$ and 
  $\varDelta = \partial^2_{x_1} + \partial^2_{x_2} +\dots +  \partial^2_{x_n}$ is the 
	usual Laplace operator in the Euclidean space ${\mathbb R}^n$,  $n\geq 2$. 
	For a differential operator $A$ acting on sections of the vector bundle
$\varLambda^{q}$ over $\mathbb{R}^n$, we denote by
   $C^{s,\lambda,\delta} _{\varLambda^{q}} \cap \mathcal{S}_A$
the space of all differential $q$-forms
   $u$ with components from $C^{s,\lambda,\delta} $, 
satisfying $Au = 0$ in the sense of the distributions in $\mathbb{R}^n$.
This space is obviously closed subspace of
   $C^{s,\lambda,\delta} _{\varLambda^{q}} $
    and so this is Banach space under the induced norm.

Let us introduce   anisotropic H\"older spaces 
which suit well to parabolic theory and are weighted at
$x = \infty$ (see
    \cite{Be11}, \cite{ShlTa18}, \cite{GaSh20}
and elsewhere).

More generally, given a Banach space $\mathcal{B}$, we denote by
$C^{s,0} ([0,T], \mathcal{B})$ the
Banach space of all mappings $v : [0,T] \to \mathcal{B}$ with finite norm
$$
   \| v \|_{C^{s,0} ([0,T], \mathcal{B})}
 = \sum_{j=0}^s \sup_{t \in [0,T]} \| (d/dt)^j v \|_{\mathcal B},
$$
where $s \in \mathbb{Z}_{\geq 0}$.
We also let
$$
   \langle v \rangle_{\lambda, [0,T], \mathcal{B}}
 = \sup_{t', t'' \in [0,T] \atop t' \neq t''}
   \frac{\|  v (t') - v (t'') \|_{\mathcal{B}}}{|t' - t''|^\lambda}
$$
and let $C^{s,\lambda} ([0,T], \mathcal{B})$ stand for the space of all functions
   $v \in C^{s,0} ([0,T], \mathcal{B})$
with finite norm
$$
   \| v \|_{C^{s,\lambda} ([0,T], \mathcal{B})}
 = \sum_{j=0}^s
   \Big( \sup_{t \in [0,T]} \| (d/dt)^j v \|_{\mathcal B}
       + \langle (d/dt)^j v \rangle_{\lambda, [0,T], \mathcal{B}} \Big).
$$

The H\"{o}lder spaces in question will be parametrised several parameters $s$, $\lambda$, $\delta$,
and $T$.
By abuse of notation we introduce the special designation $\mathbf{s} (s,\lambda,\delta)$ for the quintuple
$
   \mathbf{s} (s,\lambda,\delta) := \Big( 2s,\lambda,s,\frac{\lambda}{2},\delta \Big)
$. 
Let
$
   C^{\mathbf{s} (0,0,\delta)} _T
 = C^{0,0} ([0,T], C^{0,0,\delta} )
$
be the space of all continuous functions on ${\mathbb R}^n \times [0,T]$ with finite norm
$$
   \| u \|_{C^{\mathbf{s} (0,0,\delta)}_T }
 = \sup_{(x,t) \in {\mathbb R}^n \times [0,T]} (w (x))^\delta |u (x,t)|,
$$
and, for $0 < \lambda \leq 1$,
$$
   C^{\mathbf{s} (0,\lambda,\delta)} _T
 = C^{0,0} ([0,T], C^{0,\lambda,\delta} )\, \cap\,
   C^{0,\lambda/2} ([0,T], C^{0,0,\delta} )
$$
is the space of all continuous functions on ${\mathbb R}^n \times [0,T]$ with finite norm
\begin{equation}
   \| u \|_{C^{\mathbf{s} (0,\lambda,\delta)}_T }
 = 
   \sup_{t \in [0,T]} \| u (\cdot, t) \|_{C^{0,\lambda,\delta} }
 + \sup_{t', t'' \in [0,T] \atop t' \neq t''}
   \frac{\|u (\cdot,t') - u (\cdot,t'') \|_{C^{0,0,\delta} }}{|t'-t''|^{\lambda/2}}.
\end{equation}
Then
$C^{\mathbf{s} (s,0,\delta)}_T  = \bigcap_{j=0}^s C^{j,0} ([0,T], C^{2 (s-j),0,\delta} )$
is the space of functions on  ${\mathbb R}^n \times [0,T]$ with continuous derivatives
   $\partial^{\alpha}_x \partial^j_t u$,
for $|\alpha| + 2j \leq 2s$, and with finite norm
$$
   \| u \|_{C^{\mathbf{s} (s,0,\delta)} _T }
 = \sum_{|\alpha| + 2j \leq 2s}
   \| \partial^{\alpha}_x \partial^j_t u
   \|_{C^{\mathbf{s} (0,0,\delta+|\alpha|)}_T }.
$$
Similarly,
$$
   C^{\mathbf{s} (s,\lambda,\delta)} _T
 = \bigcap_{j=0}^s
   \! \Big( \!
         C^{j,0} ([0,T], C^{2 (s\!-\!j),\lambda,\delta} ) \cap
         C^{j,\lambda/2} ([0,T], C^{2 (s\!-\!j),0,\delta} )
   \! \Big) \!
$$
is the space of functions on ${\mathbb R}^n \times [0,T]$ with continuous partial derivatives
   $\partial^{\alpha}_x \partial^j_t u$, for $|\alpha| + 2j \leq 2s$,
and with finite norm
$$
   \| u \|_{C^{\mathbf{s} (s,\lambda,\delta)} _T}
 = \sum_{|\alpha| + 2j \leq 2s}
   \| \partial^{\alpha}_x \partial^j_t u
   \|_{C^{\mathbf{s} (0,\lambda,\delta+|\alpha|)} _T}.
$$
We also need a function space whose structure goes slightly beyond the scale of function spaces
$C^{\mathbf{s} (s,\lambda,\delta)}_T $.
Namely, given any integral $k \geq 0$, we denote by
   $C^{k, \mathbf{s} (s,\lambda,\delta)} _T$
the space of all continuous functions $u$ on ${\mathbb R}^n \times [0,T]$ whose derivatives
   $\partial^\beta_x u$
belong to $C^{\mathbf{s} (s,\lambda,\delta+|\beta|)} _T$ for all multi-indices
$\beta$ satisfying $|\beta| \leq k$, with finite norm
$$
   \| u \|_{C^{k, \mathbf{s} (s,\lambda,\delta)}_T }
 = \sum_{|\beta| \leq k}
   \| \partial^{\beta}_x u \|_{C^{\mathbf{s} (s,\lambda,\delta+|\beta|)} _T}.
$$
For $k = 0$, this space just amounts to $C^{\mathbf{s} (s,\lambda,\delta+|\beta|)}_T $,
and so we omit the index $k=0$.
The normed spaces $C^{k, \mathbf{s} (s,\lambda,\delta)}_T$ are obviously Banach spaces.

We note that the function classes introduced above can be thought of as ``physically'' 
admissible solutions to the Navier-Stokes equations (at least for proper numbers $\delta$).
By the construction, if $1\leq p<+\infty$ 
and $\delta >n/p$ then there exists a constant $c(\delta,p) > 0$ 
depending on $\delta$ and $p$, such that
\begin{equation} \label{eq.Lq}
   \| u (\cdot,t) \|_{L^p (\mathbb{R}^n)}
 \leq  c(\delta,p)\, \| u \|_{C^{\mathbf{s} (0,0,\delta)} _T }
\end{equation}
for all $t \in [0,T]$ and all $u \in C^{\mathbf{s} (0,0,\delta)}_T$.

Also, the following embedding theorem is rather expectable, see 
\cite{ShlTa18}, \cite{GaSh20}.

\begin{thm}
\label{t.emb.hoelder.t} 
Also, if   $s, s' \in \mathbb{Z}_{\geq 0}$,
   $\delta, \delta' \in \mathbb{R}_{\geq 0}$,
   $\lambda, \lambda' \in [0,1]$
and $k \in {\mathbb Z}_+$ such that 
   $s+\lambda \geq s'+\lambda'$ and
   $\delta \geq \delta'$,
then the space
   $C^{k, \mathbf{s} (s,\lambda,\delta)}_T$ is embedded continuously into
   $C^{k, \mathbf{s} (s',\lambda',\delta')}_T$.
The embedding is compact if $s + \lambda > s' + \lambda'$ and $\delta > \delta'$.
\end{thm}

We also need a standard lemma on the multiplication of functions, see \cite{ShlTa18}.

\begin{lem}
\label{l.product}
Let
   $s$, $k$ be nonnegative integers
and $\lambda \in [0,1]$.
If
   $u \in C^{k, \mathbf{s} (s,\lambda,\delta)} _T$ and 
   $v \in C^{k, \mathbf{s} (s,\lambda,\delta')}_T $, 
then the product $u v$ belongs to 
   $C^{k, \mathbf{s} (s,\lambda,\delta+\delta')} _T$ and
\begin{equation}
\label{eq.product}
   \| u v \|_{C^{k, \mathbf{s} (s,\lambda,\delta+\delta')}_T}
 \leq
   c\,
   \| u \|_{C^{k, \mathbf{s} (s,\lambda,\delta)} _T}
   \| v \|_{C^{k, \mathbf{s} (s,\lambda,\delta')} _T}
\end{equation}
with $c > 0$ a constant independent of $u$ and $v$.
\end{lem}

However we need scales of weighted H\"older spaces, 
that fit the refined structure of the
Navier-Stokes type equations. First, for
   $s, k \in \mathbb{Z}_{\geq 0}$ and
   $0 < \lambda < \lambda' < 1$,
we introduce
$$
   \mathcal{F}^{k,\mathbf{s} (s,\lambda,\lambda',\delta)} _T
 :=
   C^{k+1,\mathbf{s} (s,\lambda,\delta)}_T \cap
   C^{k,\mathbf{s} (s,\lambda',\delta)} _T.
$$
When given the norm 
$   \| u \|_{\mathcal{F}^{k,\mathbf{s} (s,\lambda,\lambda',\delta)}_T }
 :=
   \| u \|_{C^{k+1,\mathbf{s} (s,\lambda,\delta)} _T}
 + \| u \|_{C^{k,\mathbf{s} (s,\lambda',\delta)} _T}$, 
this is obviously a Banach space. The following lemma explains 
why this scale is important for our exposition, see \cite{ShlTa18}.

\begin{lem}
\label{l.mathfrak.compact}
Let
   $s$ be a positive integer,
   $k \in \mathbb{Z}_{\geq 0}$,
   $0 < \lambda < \lambda' < 1$
and
   $\delta > \delta'$.
Then the embedding 
$   \mathcal{F}^{k,\mathbf{s} (s,\lambda,\lambda',\delta)} _T
 \hookrightarrow
   \mathcal{F}^{k+1,\mathbf{s} (s-1,\lambda,\lambda',\delta')} _T
 $ is compact.
\end{lem}

Consider the induced vector bundle $\varLambda^q (t)$ 
over ${\mathbb R}^n \times [0,+\infty)$ consisting of the differential 
forms with coefficients depending on both the variable $x\in  {\mathbb R}^n$ and 
on the real parameter $t \in [0,+\infty)$. 
In the sequel we consider the following Cauchy problem. 
Given any sufficiently regular differential forms 
   $f = \sum_{\# I=q} f_I (x,t) dx_I$ and
   $u_0 = \sum_{\# I=q} u_{I,0} (x) dx_I$
on    ${\mathbb R}^n \times [0,T]$ and   ${\mathbb R}^n$,
respectively, find a pair $(u,p)$ of sufficiently regular differential forms
   $u = \sum_{\# I=q} u_I (x,t) dx_I$ and
   $p =\sum_{\# I =q-1} p_I (x,t) dx_I $
on ${\mathbb R}^n \times [0,T]$ satisfying
\begin{equation}
\label{eq.NS}
\left\{
\begin{array}{rcll}
   \partial _t u   -  \mu \varDelta u + {\mathcal N}_q u 
	+ a \, d_{q-1} p
 & =
 & f,
 & (x,t) \in {\mathbb R}^n \times (0,T),
\\
   a \, d_{q-1}^* \, u
 & =
 & 0,
 & (x,t) \in {\mathbb R}^n\times (0,T),
\\
   a \, d_{q-2}^* \, p
 & =
 & 0,
 & (x,t) \in {\mathbb R}^n\times (0,T),
\\
   u
& =
& u_0,
& (x,t) \in \mathbb{R}^n \times \{ 0 \}
\end{array}
\right.
\end{equation}
with positive fixed numbers $T$ and $\mu$, a parameter $a$ that, equals to $0$ or $1$, 
and a non-linear term ${\mathcal N}_q u$ 
that is specified by the following assumptions (see \cite{ShlTa21} or 
\cite{MShlT19} for more general problems in the context of elliptic differential complexes):
\begin{equation}
\label{eq.nonlinear}
{\mathcal N}_q u = M^{(q)}_1 (d_q \oplus d_{q-1}^* u, u) + d_{q-1} M^{(q)}_2 (u, u) 
\end{equation}
with two bilinear differential operators with constant coefficients and of zero order:
\begin{equation}
\label{eq.form.1}
 M^{(q)}_1 (v, u): C^\infty _{\varLambda^{q+1} \oplus 
\varLambda^{q-1} } (\mathbb{R}^n) \times 
C^\infty _{\varLambda^{q}} (\mathbb{R}^n) \to C^\infty _{\varLambda^{q}} (\mathbb{R}^n), 
 \end{equation}
\begin{equation}
\label{eq.form.2}
 M^{(q)}_2 (v, u): C^\infty _{\varLambda^{q}} (\mathbb{R}^n) \times 
C^\infty _{\varLambda^{q}} (\mathbb{R}^n) \to C^\infty _{\varLambda^{q-1}}  (\mathbb{R}^n).  
\end{equation}
Of course, we have to assume that $d_{q-1}^* u _0 =0$ on ${\mathbb R}^n$ if $a=1$, 
and, as we want to provide the uniqueness for 
solutions to \eqref{eq.NS}, we have to set $p=0$ if $a=0$.

For $n=1$, $q=0$ and ${\mathcal N}_0 u = u' \, u$ relations \eqref{eq.NS}  reduce  obviously 
to the Cauchy problem for  Burgers' equation, \cite{Burg40}. 

If we denote by $\star$ the $\star$-Hodge operator and by $\wedge$ 
the exterior product of differential forms then  
for $n=3$, $q=1$, $a=1$ we may identify $1$-forms with $n$-vector-fields, 
the operator $d_0$ with the gradient operator $\nabla$, 
the operator $(-d_0^*)$ with the divergence operator 
and the operator $d_1$ with the rotation operator. Then for the non-linearity
\begin{equation}
\label{eq.Lamb}
{\mathcal N}_1 u = (u \cdot \nabla) u  = \star (\star d_1 u \wedge u) + d_0 
|u|^2 /2, 
\end{equation}
written in  the Lamb form, relations \eqref{eq.NS} are usually referred to as but the Navier-Stokes 
equations for in\-com\-pressible fluid with given  dynamical viscosity $\mu$ of the fluid 
under the consideration,   density vector of outer forces $f$,   the initial velocity $u_0$
and the search-for velocity vector field $u$ and the pressure $p$ of the flow, see for 
instance \cite{Tema79}. In \cite{ShlTa21} these equations with $a=1$ 
were considered in 
Bochner-Sobolev type spaces; as it was explained there, for $q=0$ and 
$q=n$ the equations become degenerate in a sense,  so, if $a=1$
 we will consider the equations for $1\leq q\leq n-1$, only. 

Let us  comment  the example by 
\cite{PlSv03} by P. Plech\'a$\rm \check{c}$  and 
V. $\rm \check{S}$ver\'ak. 

\begin{exmp} \label{ex.SvPl}
If $\mu=1$, $q=1$, $a=0$, $b$ is a real parameter, and 
\begin{equation} \label{eq.SvPl}
{\mathcal N}_1 u
 =  (u \cdot \nabla) u \,b  + 
\frac{(1-b)\, \nabla |u|^2 + (\mathrm{div } u) u}{2}   = 
 \star (\star d_1 u \wedge u) \,b + \frac{d_0 |u|^2-(d_0^* u) u }{2}  
\end{equation}
 then \eqref{eq.NS} becomes the   
non-linear problem in ${\mathbb R}^n \times [0,T)$ considered in 
\cite{PlSv03}. 
 Actually, they consider  the  `radial vector fields' 
\begin{equation} \label{eq.radial.u}
u= - v(r,t) x,
\end{equation}
with functions $v$ of variables $t$ and $r=|x|$. Under the hypothesis of this example 
the fields are solutions to 
\eqref{eq.NS} for $f=0$ and $u_0 = - v(r,0) x$ if 
\begin{equation} \label{eq.radial.eq}
v'_t = v''_{rr} + \frac{n+1}{r} v'_r +  (n+2) v^2 +3r v v' _r . 
\end{equation}
Next,  for $v$ satisfying \eqref{eq.radial.eq} they consider  the self-similar solutions 
\begin{equation} \label{eq.self-sim.v}
v (r,t) = \frac{1}{2\varkappa (T-t)} w \Big( \frac{r}{\sqrt{2\varkappa (T-t)}} \Big)
\end{equation}
with  functions $w (y)$ binded by the following relations, see \cite[(1.9)-(1.11)]{PlSv03}: 
\begin{equation} \label{eq.self-sim}
w'' +  \frac{n+1}{y} w' -  \varkappa \, y \, w' + 
(n+2) w^2 +3y w w'  
 - \textcolor{red}{2} \varkappa \, w =0, \, y\in (0,+\infty),
\end{equation}
\begin{equation} \label{eq.self-sim.Cauchy}
w(0) = \gamma\geq 0, \, w'(0) = 0 ,\, 
w(y) = y^{-2} \mbox{ as } y \to + \infty, 
\end{equation}
with a positive parameter $\varkappa$.  
Based on some analysis of solutions to the steady equation related to 
\eqref{eq.radial.eq} and  numerical simulations, they made conclusion that  
for $n>4$ self-similar solutions  \eqref{eq.self-sim.v}  may produce singular 
solutions in finite time to  this particu\-lar version of \eqref{eq.NS} 
for regular data via formula \eqref{eq.radial.u} if $\gamma>0$. 
However it might be, the numerical simulations can not be arguments in analysis.  
Despite  claimed \cite{PlSv03}  'strong numerical evidence supporting existence of blow-up 
solutions' for compactly supported data if $n>4$, there are 
obstacles for the existence of blow-up non-periodic solutions to \eqref{eq.NS}  
with smooth data vanishing on the boundary of a bounded domain in ${\mathbb R}^n$  
(in particular, with smooth compactly supported data),  
related to the radial vector fields. 
This follows from  \cite[Lemma 2.1]{PlSv03} because according to it 
solutions to \eqref{eq.self-sim}, \eqref{eq.self-sim.Cauchy} 
are positive on $(0,+\infty)$ if $\gamma>0$. 
On the other hand, they showed that  
certain asymptotic behaviour of the initial data at the infinity 
with respect to the space variables  prevents blow-up behaviour in a finite time interval 
for the considered type of solutions, at least in the dimension $n=3$. This gives 
some hope that the use of the weighted H\"older spaces with proper weight indexes 
may exclude the blow-up behaviour of solutions 
to the Navier-Stokes type equations, at least for the non-linearity 
\eqref{eq.SvPl}. 
\end{exmp}

Thus, we will investigate the Navier-Stokes type equations \eqref{eq.NS} over 
the scale of the weighted H\"older spaces 
$\mathcal{F}^{k,\mathbf{s} (s,\lambda,\lambda',\delta)}$. 
With this purpose, for a linear operator $A: X \to Y$ between Banach spaces $X,Y$ 
with a domain ${\mathcal D}_A \subset X$ we denote by $X_{{\mathcal D}_A}$ the Banach space 
endowed with the so-called graph norm
$$
\|u \|_{X_{{\mathcal D}_A}} = \|u \|_{X} + 
\|A u  \|_{Y} \mbox{ for all } u \in {\mathcal D}_A.
$$
Thus, we introduce
   $C^{k,\mathbf{s} (s,\lambda,\delta)}_{T,\varLambda^{q}}$  
to be the space of all exterior differential $q$-forms $u$ with the coefficients 
from   $C^{k,\mathbf{s} (s,\lambda,\delta)} _T$ endowed with the natural norm.
Let also 
  $C^{k,\mathbf{s} (s,\lambda,\delta)}_{T,\varLambda^{q}, \mathcal{D} _{d\oplus d^*}}$  
	be a subset of the space $C^{k,\mathbf{s} (s,\lambda,\delta)}_{T,\varLambda^{q}}$
with the property that
   $d_q\oplus d_{q-1}^* u \in C^{k,\mathbf{s} (s,\lambda,\delta+1)}_{T,\varLambda^{q+1} 
	\oplus \varLambda^{q-1}}$;
we endow this space with  the graph norm
$$
   \| u \|_{C^{k,\mathbf{s} (s,\lambda,\delta)}_{T,\varLambda^{q},\mathcal{D} _{d\oplus d^*}} }
 = \| u \|_{C^{k,\mathbf{s} (s,\lambda,\delta)} _{T,\varLambda^{q}}}
 + \| d_q\oplus d_{q-1}^* \|_{C^{k,\mathbf{s} (s,\lambda,\delta+1)}_{T,\varLambda^{q+1} 
	\oplus \varLambda^{q-1}}}.
$$
Similarly, let
  ${\mathcal F}^{k,\mathbf{s} (s,\lambda,\lambda',\delta)}_{T,\varLambda^{q}, \mathcal{D}_
	{d\oplus d^*}}$  
	be a subset of ${\mathcal F}^{k,\mathbf{s} (s,\lambda,\lambda',\delta)}_{T,\varLambda^{q}}$
with the property that  $d_q\oplus d_{q-1}^* u \in {\mathcal F}^{k,\mathbf{s} (s,\lambda,\lambda',\delta+1)}_{T,\varLambda^{q+1} 
	\oplus \varLambda^{q-1}}$;
we endow this space with the graph norm
$$
   \| u \|_{{\mathcal F}^{k,\mathbf{s} (s,\lambda,\lambda',\delta)}_{T,\varLambda^{q},
	\mathcal{D}_{d\oplus d^*}} } = \| u \|_{{\mathcal F}^{k,\mathbf{s} (s,\lambda,\lambda',\delta)} 
	_{T,\varLambda^{q}}} + \| d_q\oplus d_{q-1}^* \|
	_{{\mathcal F}^{k,\mathbf{s} (s,\lambda,\lambda',\delta+1)}_{T,\varLambda^{q+1} 
	\oplus \varLambda^{q-1}}}.
$$

Let us continue with a suitable linearization of \eqref{eq.NS} over the 
defined scales.

\section{The Navier-Stokes type equations as Fredholm mappings}

First, we recall the notion of Fredholm mappings in Banach spaces, see \cite{Sm65}.
It is said that a linear bounded operator $A_0: X \to Y$ has the Fredholm property if 
its kernel and co-kernel are finite-dimensional subspaces of $X$ and $Y$, respectively, and 
its range $R(A)$ is closed in $Y$. Then a non-linear mapping 
$A: X \to Y$ is Fredholm if its Fr\'echet derivative $A'_{|v}$ 
is a linear bounded Fredholm operator at each point $v\in X$. The Fredholm 
property provides many uselful information on the operator equation $Au=f$ in 
the Banach spaces $X$ and $Y$, see \cite{Sm65} (cf. also \cite{Tema95} 
for the steady Navier-Stokes equations). 

We continue this section with the following linear Cauchy problem for $n\geq 2$. 
Given  any  $0\leq q \leq n$ and any sufficiently regular differential forms 
 $$w = \sum_{\# I=q} w_I (x,t) dx_I, \,  
   f = \sum_{\# I=q} f_I (x,t) dx_I, \, 
   u_0 = \sum_{\# I=q} u_{I,0} (x) dx_I
	$$
on    ${\mathbb R}^n \times [0,T]$ and   ${\mathbb R}^n$,
respectively, find a pair $(u,p)$ of sufficiently regular differential forms
   $u = \sum_{\# I=q} u_I (x,t) dx_I$ and
   $p =\sum_{\# I =q-1} p_I (x,t) dx_I $
on ${\mathbb R}^n \times [0,T]$ satisfying
\begin{equation}
\label{eq.NS.lin}
\left\{
\begin{array}{rcll}
   \partial _t u   -  \mu \varDelta u + {\mathcal B}_q (u,w) 
	+  a\, d_{q-1} p
 & =
 & f,
 & (x,t) \in {\mathbb R}^n \times (0,T),
\\
    a\, d_{q-1}^* \, u
 & =
 & 0,
 & (x,t) \in {\mathbb R}^n\times (0,T),
\\
    a \, d_{q-2}^* \, p
 & =
 & 0,
 & (x,t) \in {\mathbb R}^n\times (0,T),
\\
   u
& =
& u_0,
& (x,t) \in \mathbb{R}^n \times \{ 0 \}
\end{array}
\right.
\end{equation}
where $a\, d^*_{q-1} u _0 =0$ in ${\mathbb R}^n$ and 
${\mathcal B}_q (u,w)$ is given by 
\begin{equation}
\label{eq.linear}
M^{(q)}_1 ((d_q \oplus d_{q-1}^* u, w) + d_{q-1} M^{(q)}_2 (u, w) +
  M^{(q)}_1 ((d_q \oplus d_{q-1}^* w, u) + d_{q-1} M^{(q)}_2 (w, u) 
\end{equation}
Again, as we want to provide the uniqueness for 
solutions to  \eqref{eq.NS.lin}, we have to set $p=0$ 
if $a=0$.
 
We are moving towards  expectable uniqueness and existence theorem.
However, it depends drastically on the paparemeter $a$.
 
\begin{thm}
\label{t.NS.deriv.unique.0}
Let $n\geq 2$, $0\leq q\leq n$, $a=0$. 
Assume that
   $s, k\in \mathbb N $,
   $0 < \lambda < \lambda' < 1$,
   $\delta> n/2$, 
and 	$w \in 
	\mathcal{F}^{k,\mathbf{s} (s,\lambda,\lambda',\delta)}_{T, \varLambda^q}$.  
Then for any pair
\begin{equation} \label{eq.pair.data.0}
   F
 \! = \! (f,u_0)
 \in
   \mathcal{F}^{k,\mathbf{s} (s,\lambda,\lambda',\delta)} _{T, \varLambda^q} 
 \times
   C^{2s+k+1,\lambda,\delta}_{\varLambda^q}  
\end{equation}
there is a unique solution
$
u
 \in
   \mathcal{F}^{k,\mathbf{s} (s,\lambda,\lambda',\delta)} _{T, \varLambda^q, } 
$ 
to \eqref{eq.NS.lin} and,  moreover, 
$$
\|u\|_{\mathcal{F}^{k,\mathbf{s} (s,\lambda,\lambda',\delta)} _{T, \varLambda^q } 
 } \leq 
c (w) 
\|F\|_{ \mathcal{F}^{k,\mathbf{s} (s,\lambda,\lambda',\delta)} _{T, \varLambda^q} 
 \times
   C^{2s+k+1,\lambda,\delta}_{\varLambda^q} }
$$
with a positive constant $c (w)$ independent on $F$.
\end{thm}

\begin{proof} We use the theory of operator equations in Banach spaces and
 method of integral representation. Namely, 
Let $\psi_\mu$ be the standard fundamental solution of the convolution type to the heat operator $H_\mu = \partial _t - \mu \Delta$ in $\mathbb{R}^{n+1}$, $n\geq 1$,
$$
   \psi_\mu (x,t)
 = \frac{\theta (t)}{\left(4  \pi \mu t\right)^{n/2}}\
   e^{-\frac{\scriptstyle |x|^2}{\scriptstyle 4 \mu t}},
$$
where $\theta (t)$ is the Heaviside function. 
We set 
$$
\psi_{\mu,q} (x,y,t) = \sum_{|I|=q} \psi_\mu (x-y,t) \, (\star dy_I) \, dx_I,
$$ 
and for $q$-forms $v$ and $u_0$ over ${\mathbb R}^n \times [0,T]$ and 
${\mathbb R}^n$, respectively, denote by
\begin{eqnarray*}
   (\varPsi_\mu  v) (x,t)
 & = &
   \int_0^{t} \!\!\! \int _{\mathbb{R}^n} v (y,s) \wedge \psi_{\mu,q} (x,y, t-s) \,  ds,
\\
   (\varPsi_{\mu,q,0}  u_0) (x,t)
 & = &
   \int_{\mathbb{R}^n} u_0 (y)  \wedge \psi_{\mu,q} (x,y, t) 
\end{eqnarray*}
the so-called volume parabolic potential and Poisson parabolic potential, 
respectively, defined for $(x,t) \in {\mathbb R}^n \times (0,T) $. 

\begin{lem}
\label{l.heat.key1}
Let
   $s, k \in \mathbb{Z}_{\geq 0}$,
   $0 < \lambda< 1$ and
   $\delta > 0$.
The parabolic potentials
   $\varPsi_{\mu,q}$ and
   $\varPsi_{\mu,q,0}$
induce bounded linear operators
$$
 \varPsi_{\mu,q,0} :
 C^{2s+k,\lambda,\delta}_{\varLambda^q} (\mathbb{R}^n)
 \to
 C^{k,\mathbf{s} (s,\lambda,\delta)} _{T,\varLambda^q}  \cap \mathcal{S}_{H_\mu}, 
$$
$$
   \varPsi_{\mu,q} :
  C^{k,\mathbf{s} (s,\lambda,\delta)}_{T,\varLambda^q} 
  \to
  C^{k,\mathbf{s} (s,\lambda,\delta)} _{T,\varLambda^q, \mathcal{D}_{H_\mu}}  ,
\,\, 
  \varPsi_{\mu,q} :
  C^{k,\mathbf{s} (s-1,\lambda,\delta+2)}_{T,\varLambda^q} 
  \to
  C^{k,\mathbf{s} (s,\lambda,\delta)} _{T,\varLambda^q}  .
$$
\end{lem}

\begin{proof} As the potentials act  on the differential forms coefficient-wise, 
the statement follows from \cite[Lemmas 4.5 and 4.8]{ShlTa18}. 
\end{proof}

Now we set 
$$
W_q u = 	 M^{(q)}_1 ((d_q \oplus d_{q-1}^*) u, w) +   M^{(q)}_1 (
(d_q \oplus d_{q-1}^*) w, u  )  .
$$

\begin{lem} \label{l.compact}
If $k\geq {\mathbb N}$ and $\delta>1$ then following the operators are compact: 
\begin{equation}
\label{eq.compact}
\varPsi_{\mu,q} B_q (w,\cdot): {\mathcal F}^{k,\mathbf{s} (s,\lambda,\delta)} 
_{T ,\varLambda^{q}}
\to  
{\mathcal F}^{k,\mathbf{s} (s,\lambda,\delta)} _{T ,\varLambda^{q}  }, \,\,
\varPsi_{\mu,q} W_q: {\mathcal F}^{k,\mathbf{s} (s,\lambda,\delta)} _{T ,\varLambda^{q}, 
\mathcal{D}_{d\oplus d^*}}
\to  
{\mathcal F}^{k,\mathbf{s} (s,\lambda,\delta)} _{T ,\varLambda^{q} 
\mathcal{D}_{d\oplus d^*} }
\end{equation}
\end{lem}

\begin{proof}
 According to 
embedding Theorem \ref{t.emb.hoelder.t}, multiplication Lemma \ref{eq.product}, 
and Lem\-ma \ref{l.heat.key1}, the operators 
\begin{equation} \label{eq.cont.lin.0}
\varPsi_{\mu,q} B_q(w,\cdot): C^{k,\mathbf{s} (s,\lambda,\delta)} _{T ,\varLambda^{q}, }
\to  
C^{k-1,\mathbf{s} (s+1,\lambda,2\delta-1)} _{T ,\varLambda^{q}} ,
\end{equation}
\begin{equation} \label{eq.cont.lin.1}
\varPsi_{\mu,q} W_q: C^{k,\mathbf{s} (s,\lambda,\delta)} _{T ,\varLambda^{q}, 
\mathcal{D}_{d\oplus d^*}}
\to  
C^{k,\mathbf{s} (s+1,\lambda,2\delta-1)} _{T ,\varLambda^{q}} ,
\end{equation}
\begin{equation} \label{eq.cont.lin.2}
d_q \oplus d_{q-1}^*
\varPsi_{\mu,q} W_q: C^{k,\mathbf{s} (s,\lambda,\delta)} _{T ,\varLambda^{q}, 
\mathcal{D}_{d\oplus d^*}}
\to  
C^{k-1,\mathbf{s} (s+1,\lambda,2\delta)} _{T ,\varLambda^{q+1} \oplus 
\varLambda^{q-1}} ,
\end{equation}
are continuous if $k\geq 1$, $\delta>0$. As the embeddings 
\begin{equation} \label{eq.emb.1}
C^{k,\mathbf{s} (s+1,\lambda,2\delta-1)} _{T ,\varLambda^{q}} \to 
C^{k+2,\mathbf{s} (s,\lambda,2\delta-1)} _{T ,\varLambda^{q}}, 
\, C^{k-1,\mathbf{s} (s+1,\lambda,2\delta)} _{T ,\varLambda^{q+1} \oplus 
\varLambda^{q-1}} \to 
C^{k+1,\mathbf{s} (s,\lambda,2\delta)} _{T ,\varLambda^{q+1} \oplus 
\varLambda^{q-1}}
\end{equation}
are continuous, we see that the operator 
$\varPsi_{\mu,q} W_q$ maps the space 
$C^{k,\mathbf{s} (s,\lambda,\delta)} _{T ,\varLambda^{q}, 
\mathcal{D}_{d\oplus d^*}}$ continuously  to  
$C^{k+1,\mathbf{s} (s,\lambda,2\delta-1)} _{T ,\varLambda^{q} 
\mathcal{D}_{d\oplus d^*} } $ and the operator 
$\varPsi_{\mu,q} B_q (w,\cdot)$ maps the space 
$C^{k,\mathbf{s} (s,\lambda,\delta)} _{T ,\varLambda^{q}}$ continuously  to  
$C^{k+1,\mathbf{s} (s,\lambda,2\delta-1)} _{T ,\varLambda^{q} } $. 
In particular, the operators  
$$
\varPsi_{\mu,q} B_q (w,\cdot): {\mathcal F}^{k,\mathbf{s} (s,\lambda,\delta)} 
_{T ,\varLambda^{q}}
\to  
{\mathcal F}^{k+1,\mathbf{s} (s,\lambda,2\delta-1)} _{T ,\varLambda^{q}  }, 
$$
$$ 
\varPsi_{\mu,q} W_q: {\mathcal F}^{k,\mathbf{s} (s,\lambda,\delta)} _{T ,\varLambda^{q}, 
\mathcal{D}_{d\oplus d^*}}
\to  
{\mathcal F}^{k+1,\mathbf{s} (s,\lambda,2\delta-1)} _{T ,\varLambda^{q} 
\mathcal{D}_{d\oplus d^*} }
$$
are continuous, too, for $k \in \mathbb N$, $\delta>0$. If $k\in \mathbb N$, 
$\delta>1$ then $2\delta-1>\delta$ and hence, by Lemma 
\ref{l.mathfrak.compact}, the operators \eqref{eq.compact} are compact. 
\end{proof}

Next we reduce the Cauchy problem \eqref{eq.NS.lin} to an operator Frredholm 
equation.

\begin{lem}
\label{l.invertible.psi.lin.0}
Let $0\leq q \leq n$, 
   $s$ and $k $ be positive integers,
   $0 < \lambda < \lambda' < 1$,
   $\delta \in (n/2,+\infty)$, 
and $w \in \mathcal{F}^{k,\mathbf{s} (s,\lambda,\lambda',\delta)} 
_{T, \varLambda^{q}} $. 
Then the operator
\begin{equation}
\label{eq.heat.pseudo.d2.0}
   I \! + \! \varPsi_{\mu,q} B_q (w,\cdot) :
   \mathcal{F}^{k,\mathbf{s} (s,\lambda,\lambda',\delta)} 
_{T, \varLambda^{q}} 
 \to   \mathcal{F}^{k,\mathbf{s} (s,\lambda,\lambda',\delta)} 
_{T, \varLambda^{q}} 
\end{equation}
is continuously invertible. 
\end{lem}

\begin{proof} By Lemma \ref{l.compact}, the operator 
\begin{equation*}
  \varPsi_{\mu,q} B_q (w,\cdot)  :
   \mathcal{F}^{k,\mathbf{s} (s,\lambda,\lambda',\delta)} 
_{T, \varLambda^{q}} 
 \to   \mathcal{F}^{k,\mathbf{s} (s,\lambda,\lambda',\delta)} 
_{T, \varLambda^{q}} 
\end{equation*}
is compact. Hence the mapping \eqref{eq.heat.pseudo.d2.0} is a 
Fredholm linear operator of index zero by the famous Fredholm theorem; 
in particular, it is continuously invertible if and only if it is injective. 

Assume that $u \in \mathcal{F}^{k,\mathbf{s} (s,\lambda,\lambda',\delta)} 
_{T, \varLambda^{q}} $ and 
\begin{equation*}
u \! + \! \varPsi_{\mu,q} \, B_q (w,u)  =0.
\end{equation*}
 Then the properties 
of the fundamental solution $\varPsi_\mu$ mean that $u$ is a solution 
to the following Cauchy problem:
\begin{equation*} 
\left\{
\begin{array}{llll}
H_\mu u \! + \! B_q (w,u) & = & 0  & (x,t) \in {\mathbb R}^n \times (0,T), 
\\
   u
& =
& 0,
& (x,t) \in \mathbb{R}^n \times \{ 0 \}.
\end{array}
\right.
\end{equation*}
In particular, \eqref{eq.deRham} and an integration by parts yields for all $t \in [0,T]$:
\begin{equation} \label{eq.energy.lin.1}
\partial_t \|u (\cdot, t)\|^2_{L^2_{\varLambda^q} ({\mathbb R}^n)} 
+ \mu \, \sum_{j=1}^n\|\partial_j u (\cdot, t)\|^2_{L^2_{\varLambda^{q}} ({\mathbb R}^n)} 
 = (B_q (w,u)(\cdot, t),u (\cdot, t))^2_{L^2_{\varLambda^{q}} ({\mathbb R}^n)} .
\end{equation}

Using the structure of the operator $B_q (w,\cdot)$, see see that 
there are positive constants $c_q ^{(j)}$ independent on $t$ and $u$ such that 
\begin{equation} \label{eq.energy.lin.2}
|(B_q (w,u) (\cdot, t),u (\cdot, t))^2_{L^2_{\varLambda^{q}} ({\mathbb R}^n)} | \leq 
c_q ^{(1)}\|\nabla w\|_{C^{0,\mathbf{s} (0,0,\delta+1)}_{T,\varLambda^{q+1}}}  
\| u (\cdot, t)\|^2_{L^2_{\varLambda^{q}} ({\mathbb R}^n)} +
\end{equation}
$$
c_q ^{(2)}\|w\|_{C^{0,\mathbf{s} (0,0,\delta)}_{T,\varLambda^{q}}}  
\|\nabla u(\cdot, t) \|_{L^2_{\varLambda^{q}} ({\mathbb R}^n)} 
\| u(\cdot, t)\|_{L^2_{\varLambda^{q}} ({\mathbb R}^n)} \leq 
$$
$$\frac{\mu}{2} \|\nabla u(\cdot, t) \|^2_{L^2_{\varLambda^{q}} ({\mathbb R}^n)} +
\Big( 
\frac{2 (c_q ^{(2)})^2}{\mu} \|w\|^2_{C^{0,\mathbf{s} (0,0,\delta)}_{T,\varLambda^{q}}}
+ c_q ^{(1)}\|\nabla w\|_{C^{0,\mathbf{s} (0,0,\delta+1)}_{T,\varLambda^{q}}}  \Big)
\| u(\cdot, t)\|^2_{L^2_{\varLambda^{q}} ({\mathbb R}^n)} 
$$
for all $t \in [0,T]$. Now, combining \eqref{eq.energy.lin.1} and 
\eqref{eq.energy.lin.2}  we conclude that for all $t \in [0,T]$  
\begin{equation*} 
\partial_t \|u (\cdot, t)\|^2_{L^2_{\varLambda^q} ({\mathbb R}^n)}
\leq C_\mu \| u(\cdot, t)\|^2_{L^2_{\varLambda^{q}} ({\mathbb R}^n)} 
\end{equation*}
with  a positive constant $C_\mu$ independent on $u$. 
Finally, the Gronwall lemma, see, for instance, 
\cite[Ch. XII, p. 353, formula (1.2)]{MPF91}
 yields that $u\equiv 0$, that was to be proved. 
\end{proof}

Finally, according to Lemma \ref{l.heat.key1}, the form  
$$
v^{(0)} =\varPsi_{\mu,q}  f +  \varPsi_{\mu,q,0}  u_0 ,
$$
associated with  the pair \eqref{eq.pair.data}, belongs to the space 
$\mathcal{F}^{k,\mathbf{s} (s,\lambda,\lambda',\delta)} 
_{T, \varLambda^{q}} $.   
Then, 
there is a unique form $u \in \mathcal{F}^{k,\mathbf{s} (s,\lambda,\lambda',\delta)} 
_{T, \varLambda^{q}} $, satisfying 
 \begin{equation} \label{eq.pseudo.lin.0} 
u \! + \! \varPsi_{\mu,q} B_q(w,u) =v^{(0)}.
\end{equation}
By the properties of the fundamental solution $\varPsi_{\mu}$, we have 
\begin{equation}  \label{eq.Cauchy.lin}
\left\{
\begin{array}{llll}
H_\mu u \! + \!  B_q (w, u)    & = & 
f  & (x,t) \in {\mathbb R}^n \times (0,T), 
\\
   u
& =
& u_0,
& (x,t) \in \mathbb{R}^n \times \{ 0 \}.
\end{array}
\right.
\end{equation}

As we have seen in the proof of Lemma \ref{l.invertible.psi.lin.0}, 
problem \eqref{eq.Cauchy.lin} has no more than one solution 
in the weighted H\"older spaces and 
that was to be proved.
\end{proof}

Next, we consider the case where  $a=1$. 
At the degree $q=1$ the following theorem was proved in \cite[Corollary 5.9]{ShlTa18}.

\begin{thm}
\label{t.NS.deriv.unique}
Let $n\geq 2$, $1\leq q\leq n-1$, $a=1$. 
Assume that
   $s, k\in \mathbb N $,
   $0 < \lambda < \lambda' < 1$,
   $n/2<\delta <n$, \textcolor{red}{$\delta \ne (n-1)$}  
and 	$w \in 
	\mathcal{F}^{k,\mathbf{s} (s,\lambda,\lambda',\delta)}_{T, \varLambda^q, {\mathcal D}
	_{d\oplus d^*}}$. 
Then for any pair
\begin{equation} \label{eq.pair.data}
   F
 \! = \! (f,u_0)
 \in
   \mathcal{F}^{k,\mathbf{s} (s,\lambda,\lambda',\delta)} _{T, \varLambda^q, {\mathcal D}
	_{d\oplus d^*}} 
 \times
   C^{2s+k+1,\lambda,\delta}_{\varLambda^q}  \cap \mathcal{S}_{ d^\ast}
\end{equation}
there is a unique solution
$$
   U
 \! = \! (u,p)
 \in
   \mathcal{F}^{k,\mathbf{s} (s,\lambda,\lambda',\delta)} _{T, \varLambda^q, {\mathcal D}
	_{d\oplus d^*}} 
 \times
   \mathcal{F}^{k-1,\mathbf{s} (s-1,\lambda,\lambda',\delta-1)} _{T, \varLambda^{q-1}, 
	{\mathcal D}_{d\oplus d^*}} , 
$$
to \eqref{eq.NS.lin} and,  
moreover, 
$$
\|U\|_{\mathcal{F}^{k,\mathbf{s} (s,\lambda,\lambda',\delta)} _{T, \varLambda^q, 
 {\mathcal D}	_{d\oplus d^*}} 
 \times
   \mathcal{F}^{k-1,\mathbf{s} (s-1,\lambda,\lambda',\delta-1)} _{T, \varLambda^{q-1}, 
{\mathcal D}_{d\oplus d^*}} } \leq 
c (w) 
\|F\|_{ \mathcal{F}^{k,\mathbf{s} (s,\lambda,\lambda',\delta)} _{T, \varLambda^q, 
{\mathcal D}	_{d\oplus d^*}} 
 \times
   C^{2s+k+1,\lambda,\delta}_{\varLambda^q} }
$$
with a positive constant $c (w)$ independent on $F$.
\end{thm}

\begin{proof} Let 
$$
   e (x)
 = \left\{ \begin{array}{lcl}
             \displaystyle
             \frac{1}{\pi} \ln |x|,
           & \mbox{for}
           & n = 2,
\\
             \displaystyle
             \frac{1}{\sigma_n} \frac{|x|^{2-n}}{2-n},
           & \mbox{for}
           & n \geq 3,
           \end{array}
   \right.
$$
be the standard two-sided fundamental solution of the convolution type to the Laplace operator
in $\mathbb{R}^n$ and $\sigma_n$ the area of the unit sphere in $\mathbb{R}^n$.
We set 
$$
e_q (x,y) = \sum_{|I|=q} e (x-y) \, (\star dy_I) \, dx_I,
$$
and 
then, for $f \in C^{k,\mathbf{s} (s,\lambda,\delta)} _{T ,\varLambda^{q+1}} $, 
\begin{equation} 
\label{eq.volume.potential.d} 
(\varPhi _q \, f) (x,t) =
\int_{{\mathbb R}^n} f (y,t)  \wedge \phi_q (x,y)
\end{equation}
where
$
\phi_q (x,y)= (d_{n-q-1})^*_y e_q (x,y) $, $n\geq 2$.

\begin{lem} \label{l.deRham.Hoelder.t.new}
Let $n\geq 2$, $q \geq 0$, $s \in {\mathbb Z}_+$, $k \in {\mathbb Z}_+$, 
$0<\lambda < 1$, $\delta>0$, $\delta+1-n \not \in {\mathbb Z}_+$. 
The differential $d$ induces a bounded linear  operator
\begin{equation*} 
   d_q :\,
   {\mathcal F}^{k,\mathbf{s} (s,\lambda,\delta)} _{T ,\varLambda^{q}, \mathcal{D}_{d\oplus d^*}} 
	\cap \mathcal{S}_{d^*} \to
   {\mathcal F}^{k,\mathbf{s} (s,\lambda,\delta+1)} _{T ,\varLambda^{q}}
	\cap \mathcal{S}_{d}.
\end{equation*}
The related operator equation a normally solvable map; more precisely,  
\begin{enumerate}
\item 
the operator $d_q$ is an isomorphism if 
$0 < \delta<n-1 $ and its inverse is given by the integral operator $\varPhi_q$; 
\item 
if there is $m \in {\mathbb Z}_+$ such that $n-1 + m < \delta< 
n  + m$ then $d$ defines an injection and its (closed) range 
$R^{k,\mathbf{s} (s,\lambda,\delta\!+\!1)}_{T,\varLambda^{q+1}}$ consists of  $f \in 
	{\mathcal F}^{k,\mathbf{s} (s,\lambda,\delta+1)} _{T ,\varLambda^{q+1}}
	\cap \mathcal{S}_{d_{q+1}}$ satisfying
 for all  $t \in [0,T]$ and  all  $h \in H_{\leq m+1,\varLambda^{q}}$
$$
(f (\cdot,t),d_q h)_{L^2 ({\mathbb R}^n, \varLambda^{q+1})} = 0 ,
$$
and the left inverse  of $d$ is given by the integral operator $\varPhi_q$. 
\end{enumerate} 
\end{lem}

\begin{proof} Follows immediately from the definition of the scale
${\mathcal F}^{k,\mathbf{s} (s,\lambda,\delta)} _{T ,\varLambda^{q}}$ and 
\cite[Theorem 3.4]{Ga22} 
where the range of the operator 
\begin{equation*} 
   d_q \oplus d_{q-1}^*:\,
   C^{k,\mathbf{s} (s,\lambda,\delta)} _{T ,\varLambda^{q}, \mathcal{D}_{d\oplus d^*}} 
\to
   C^{k,\mathbf{s} (s,\lambda,\delta+1)} _{T ,\varLambda^{q+1}}
	\times C^{k,\mathbf{s} (s,\lambda,\delta+1)} _{T ,\varLambda^{q-1}}
\end{equation*}
was described. 
\end{proof}

Now, if $0 <\delta <n-1$ then $\varPhi_q d_q $ maps 
the space ${\mathcal F}^{k,\mathbf{s} (s,\lambda,\delta)} _{T ,\varLambda^{q} 
\mathcal{D}_{d\oplus d^*} }$ continuously to 
${\mathcal F}^{k,\mathbf{s} (s,\lambda,\delta)} _{T ,\varLambda^{q} 
\mathcal{D}_{d\oplus d^*} } \cap {\mathcal S}_{d^*}$
 for $\delta+1-n \not \in {\mathbb Z}_+$.  
If $n-1<\delta<n$, then, 
\begin{equation} \label{eq.range.D}
(d_q v (\cdot,t),d_q h)_{L^2 ({\mathbb R}^n, \varLambda^{q+1})} =
(v (\cdot,t),(d_q^*  d_q h)_{L^2 ({\mathbb R}^n, \varLambda^{q+1})} =0
\end{equation}
for all  $t \in [0,T]$ and  all  $h \in H_{\leq 1,\varLambda^{q}}$ 
and any $v\in {\mathcal F}^{k,\mathbf{s} (s,\lambda,\delta)} _{T ,\varLambda^{q}, 
\mathcal{D}_{d\oplus d^*}} $. 
Applying  Lemma \ref{l.deRham.Hoelder.t.new}
with $m=0$ we see that the operator 
\begin{equation} \label{eq.Phiq.cont}
\varPhi_q d_q: {\mathcal F}^{k,\mathbf{s} (s,\lambda,\delta)} _{T ,\varLambda^{q} 
\mathcal{D}_{d\oplus d^*} } 
\to {\mathcal F}^{k,\mathbf{s} (s,\lambda,\delta)} _{T ,\varLambda^{q} 
\mathcal{D}_{d\oplus d^*} } \cap {\mathcal S}_{d^*}
\end{equation}
is a continuous  if $n/2<\delta<n$, $\delta+1-n \not \in {\mathbb Z}_+$, too; 
in particular, $\varPhi_q d_q u = u$ for 
all ${\mathcal F}^{k,\mathbf{s} (s,\lambda,\delta)} _{T ,\varLambda^{q}, 
\mathcal{D}_{d\oplus d^*}} $. Actually, the operator $\varPhi_q d_q$
represents the Leray-Helmholtz type $L^2$-projection on the subspace 
$L^2_{\varLambda^q} ({\mathbb R}^n) \cap {\mathcal S}_{d^*}$ of 
$L^2_{\varLambda^q} ({\mathbb R}^n)$, see \cite[Corollary 1]{GaSh19} 
for the isotropic weighted H\"older spaces or 
\cite[Corollary 2]{GaSh20} for the anisotropic ones.

\begin{lem}
\label{l.invertible.psi.lin}
Let $1\leq q \leq n-1$, 
   $s$ and $k $ be positive integers,
   $0 < \lambda < \lambda' < 1$,
   $\delta \in (n/2,n)$, \textcolor{red}{$\delta \ne (n-1)$} 
and $w \in \mathcal{F}^{k,\mathbf{s} (s,\lambda,\lambda',\delta)} 
_{T, \varLambda^{q}, {\mathcal D}_{d\oplus d^*}} $. 
Then the operator
\begin{equation}
\label{eq.heat.pseudo.d2}
   I \! + \! \varPhi_q \, d_q \, \varPsi_{\mu,q} W_q :
   \mathcal{F}^{k,\mathbf{s} (s,\lambda,\lambda',\delta)} 
_{T, \varLambda^{q}, {\mathcal D}_{d\oplus d^*}} \cap \mathcal{S}_{d^*}
 \to   \mathcal{F}^{k,\mathbf{s} (s,\lambda,\lambda',\delta)} 
_{T, \varLambda^{q}, {\mathcal D}_{d\oplus d^*}} \cap \mathcal{S}_{d^*}
\end{equation}
is continuously invertible. 
\end{lem}

\begin{proof} By Lemma \ref{l.compact} and the discussion above, the operator 
\begin{equation*}
   \varPhi_q \, d_q \, \varPsi_{\mu,q} W_q :
   \mathcal{F}^{k,\mathbf{s} (s,\lambda,\lambda',\delta)} 
_{T, \varLambda^{q}, {\mathcal D}_{d\oplus d^*}} \cap \mathcal{S}_{d^*}
 \to   \mathcal{F}^{k,\mathbf{s} (s,\lambda,\lambda',\delta)} 
_{T, \varLambda^{q}, {\mathcal D}_{d\oplus d^*}} \cap \mathcal{S}_{d^*}
\end{equation*}
is compact. Hence the mapping \eqref{eq.heat.pseudo.d2} is a 
Fredholm linear operator of index zero by the famous Fredholm theorem; 
in particular, it is continuously invertible if and only if it is injective. 

First, we note that, as the scalar $H_\mu$ commutes with the differential operator $d_q$ we 
conclude that the operators $\varPhi_q d_q$ and $H_\mu$ commute, too. Then 
any element  $u$ from the kernel of the operator \eqref{eq.heat.pseudo.d2} is a solution 
to the following Cauchy problem:
\begin{equation*} 
\left\{
\begin{array}{llll}
H_\mu u \! + \! \varPhi_q \, d_q  \, 
W_q (w,u) & = & 0  & (x,t) \in {\mathbb R}^n \times (0,T), 
\\
   u
& =
& 0,
& (x,t) \in \mathbb{R}^n \times \{ 0 \}.
\end{array}
\right.
\end{equation*}

Next, according, to \cite[Corollary 2]{GaSh20}, the operator $\varPhi_q d_q$
represents the Leray-Helmholtz type $L^2$-projection on the subspace 
$L^2_{\varLambda^q} ({\mathbb R}^n) \cap {\mathcal S}_{d^*}$ of 
$L^2_{\varLambda^q} ({\mathbb R}^n)$ and hence  
$$
(B_q (w,u)(\cdot, t),u (\cdot, t))^2_{L^2_{\varLambda^{q}} ({\mathbb R}^n)}  = 
(\varPhi_q d_q W_q u(\cdot, t),u (\cdot, t))^2_{L^2_{\varLambda^{q}} ({\mathbb R}^n)} 
$$
for all $t \in [0,T]$ and all $u \in \mathcal{F}^{k,\mathbf{s} (s,\lambda,\lambda',\delta)} 
_{T, \varLambda^{q}, {\mathcal D}_{d\oplus d^*}} \cap \mathcal{S}_{d^*}$. 
Therefore, the  injectivity  of the operator \eqref{eq.heat.pseudo.d2}  follows 
in the same way as for the  operator \eqref{eq.heat.pseudo.d2.0}.
\end{proof}

To finish the proof of the theorem, we note that 
the form  
$$v^{(0)} = \varPhi_q d_q \Big(\varPsi_{\mu,q}  f +  \varPsi_{\mu,q,0}  u_0 \Big),
$$
associated with  the pair \eqref{eq.pair.data}, belongs to the space 
$\mathcal{F}^{k,\mathbf{s} (s,\lambda,\lambda',\delta)} 
_{T, \varLambda^{q}, {\mathcal D}_{d\oplus d^*}} \cap \mathcal{S}_{d^*}$ if 
$\delta \in (n/2,n)$, $\delta+1-n \not \in {\mathbb Z}_+$. 
Then, according to Lemma \ref{l.deRham.Hoelder.t.new}, 
there is a unique form $u \in \mathcal{F}^{k,\mathbf{s} (s,\lambda,\lambda',\delta)} 
_{T, \varLambda^{q}, {\mathcal D}_{d\oplus d^*}} \cap \mathcal{S}_{d^*}$, satisfying 
 \begin{equation} \label{eq.pseudo.lin} 
u \! + \! \varPhi_q \, d_q \, \varPsi_{\mu,q} W_q u =v^{(0)}.
\end{equation}
By the discussion above, see \eqref{eq.range.D} and Lemma \ref{l.deRham.Hoelder.t.new}, 
 \begin{equation} \label{eq.pseudo.lin.1} 
\varPhi_q \, d_q \, \varPsi_{\mu,q} W_q u = \varPhi_q \, d_q \, \varPsi_\mu B_q (w,u),
\end{equation}
$$
d_q (I- \varPhi_q \, d_q)  \Big(\varPsi_{\mu,q}  f +  \varPsi_{\mu,q,0}  u_0 
- \varPsi_{\mu,q} B_q (w,u)\Big) =0
$$
if 
$\delta \in (n/2,n)$, $\delta+1-n \not \in {\mathbb Z}_+$. Hence, applying statement 
(1) of Lemma \ref{l.deRham.Hoelder.t.new} we see that there is a unique form 
$\tilde p   \in
     \mathcal{F}^{k-1,\mathbf{s} (s,\lambda,\lambda',\delta-1)} _{T, \varLambda^{q-1}, 
	{\mathcal D}_{d\oplus d^*}} \cap {\mathcal S}_{d^*}$ satisfying 
 \begin{equation} \label{eq.pseudo.lin.2} 
	d_{q-1} \tilde p = (I- \varPhi_q \, d_q)  \Big(\varPsi_{\mu,q}  f +  \varPsi_{\mu,q,0}  u_0 
- \varPsi_\mu B_q (w,u)\Big). 
\end{equation}
Taking in account \eqref{eq.pseudo.lin},
\eqref{eq.pseudo.lin.1}, \eqref{eq.pseudo.lin.2}, we conclude that 
\begin{equation*}  
u \! + \!  \varPsi_{\mu,q} B_q (w,u) + d_{q-1}\tilde  p  =
\varPsi_{\mu,q}  f +  \varPsi_{\mu,q,0}  u_0
\end{equation*}
Then the form $p = H_\mu \tilde  p$ belongs to the space 
$\mathcal{F}^{k-1,\mathbf{s} (s-1,\lambda,\lambda',\delta-1)} _{T, \varLambda^{q-1}, 
	{\mathcal D}_{d\oplus d^*}} \cap {\mathcal S}_{d^*}$. 
Again, the properties 
of the fundamental solution $\varPsi_\mu$ mean that the pair $(u,p) $ is a solution 
to the Cauchy problem \eqref{eq.NS.lin}. Moreover $u$ is a solution to 
\begin{equation}  \label{eq.Cauchy.pseudo.lin}
\left\{
\begin{array}{llll}
H_\mu u \! + \!  \varPhi_q d_q 
B_q (w, u)    & = & \varPhi_q d_q 
f  & (x,t) \in {\mathbb R}^n \times (0,T), 
\\
   u
& =
& 
u_0,
& (x,t) \in \mathbb{R}^n \times \{ 0 \}.
\end{array}
\right.
\end{equation}

As we have noted in the proof of Lemma \ref{l.invertible.psi.lin}, 
problem \eqref{eq.Cauchy.pseudo.lin} has no more than one solution 
in the weighted H\"older spaces and then the uniqueness of the solution  
$(u,p)$ to \eqref{eq.NS.lin} follows from Lemma \ref{l.deRham.Hoelder.t.new} because
$d_{q-1}  p = (I- \varPhi_q \, d_q)  \Big(  f  - B_q (w,u)\Big)$, 
that was to be proved.
\end{proof}

Now we may pass to the non-linear problem \eqref{eq.NS}. 

\begin{cor}
\label{c.invertible.psi.0}
Let $n\geq 2$, $0\leq q \leq n$, $a=0$,   
   $s$ and $k $ be positive integers,
   $0 < \lambda < \lambda' < 1$,
   $\delta \in (n/2,+\infty)$. 
Then the non-linear mapping
\begin{equation}
\label{eq.heat.pseudo.nl1.0}
 \varPsi_{\mu,q}  \mathcal{N}_q :
   \mathcal{F}^{k,\mathbf{s} (s,\lambda,\lambda',\delta)} 
_{T, \varLambda^{q}} 
 \to   \mathcal{F}^{k,\mathbf{s} (s,\lambda,\lambda',\delta)} 
_{T, \varLambda^{q}} 
\end{equation} 
is continuous and compact and the mapping 
\begin{equation}
\label{eq.heat.pseudo.nl2.0}
   I \! + \! \varPsi_ {\mu,q} \mathcal{N}_q :
   \mathcal{F}^{k,\mathbf{s} (s,\lambda,\lambda',\delta)} 
_{T, \varLambda^{q}}
 \to   \mathcal{F}^{k,\mathbf{s} (s,\lambda,\lambda',\delta)} 
_{T, \varLambda^{q}} 
\end{equation}
is continuous, Fredholm, injective and open. 
\end{cor}

\begin{proof} Since    the bilinear form $\mathcal{B}_q$ is symmetric and
   $\mathcal{B}_q (u,u) = 2 \mathcal{N}_q (u)$,
we easily obtain
\begin{equation}
\label{eq.M.diff}
   \mathcal{N}_q (u') - \mathcal{N}_q (u'')
 = \mathcal{B}_q (u'', u'-u'') + (1/2)\, \mathcal{B}_q (u'-u'', u'-u'').
\end{equation}
Then, using theorem \ref{t.emb.hoelder.t}, multiplication Lemma \ref{eq.product}, 
and Lemma \ref{l.heat.key1}, cf. \eqref{eq.cont.lin.0}, we see that 
$$
\| \varPsi_{\mu,q}  
\mathcal{N}_q (u') - 
\varPsi_{\mu,q}  \mathcal{N}_q (u'')\|_{C^{k-1,\mathbf{s} (s+1,\lambda,2\delta-1)} _{T ,\varLambda^{q}}} \leq 
$$
$$
C_1 \, 
\|u''\|_{C^{k,\mathbf{s} (s,\lambda,\delta)} _{T ,\varLambda^{q}, 
\mathcal{D}_{d\oplus d^*}}} \, 
\|u'-u''\|_{C^{k,\mathbf{s} (s,\lambda,\delta)} _{T ,\varLambda^{q}}} + 
C_2 \, 
\|u'-u''\|^2_{C^{k,\mathbf{s} (s,\lambda,\delta)} _{T ,\varLambda^{q}}},
$$
with positive constants $C_j$ independent on $u',u''$. As the 
embeddings \eqref{eq.emb.1} are continuous, 
the nonlinear operator $\varPsi_{\mu,q} \mathcal{N}_q$ maps the space 
$C^{k,\mathbf{s} (s,\lambda,\delta)} _{T ,\varLambda^{q}}$ continuously  to  
$C^{k+1,\mathbf{s} (s,\lambda,2\delta-1)} _{T ,\varLambda^{q} } $. In particular, the operator 
$$
\varPsi_{\mu,q} \mathcal{N}_q: {\mathcal F}^{k,\mathbf{s} (s,\lambda,\delta)} _{T ,\varLambda^{q}, 
\mathcal{D}_{d\oplus d^*}}
\to  
{\mathcal F}^{k+1,\mathbf{s} (s,\lambda,2\delta-1)} _{T ,\varLambda^{q} 
\mathcal{D}_{d\oplus d^*} }, 
$$
is continuous, too, for $k \in \mathbb N$. If $\delta>1$ then 
$2\delta-1>\delta$ and hence, by Lemma 
\ref{l.mathfrak.compact}, 
the operator 
$$
\varPsi_{\mu,q} \mathcal{N}_q: {\mathcal F}^{k,\mathbf{s} (s,\lambda,\delta)} _{T ,\varLambda^{q}, \mathcal{D}_{d\oplus d^*}}
\to  
{\mathcal F}^{k,\mathbf{s} (s,\lambda,\delta)} _{T ,\varLambda^{q} 
\mathcal{D}_{d\oplus d^*} }
$$
is compact. 

Equality \eqref{eq.M.diff}  makes it evident that the Frech\'et derivative
$(I \! + \!  \varPsi_ {\mu,q} \mathcal{N}_q)'_{|w}$ of the nonlinear mapping 
$ (I \! + \!  \varPsi_ {\mu,q} \mathcal{N}_q)$ at an arbitrary point 
$  w  \in {\mathcal F}^{k,\mathbf{s} (s,\lambda,\delta)} _{T ,\varLambda^{q}}$ coincides with the continuous linear mapping $ (I \! + 
\!  \varPsi_ {\mu,q} B_q(w,\cdot)$. By Lemma \ref{l.invertible.psi.lin.0}, 
$(I \! + 
\!  \varPsi_ {\mu,q} B_q(w,\cdot))$ is an invertible continuous linear
mapping of the space ${\mathcal F}^{k,\mathbf{s} (s,\lambda,\delta)} _{T ,\varLambda^{q}}$ and hence the non-linear mapping \eqref{eq.heat.pseudo.nl2}
is Fredholm one. Both the openness and the injectivity of the mapping 
\eqref{eq.heat.pseudo.nl2}  follow now from the implicit function theorem 
for Banach spaces,   see for instance  \cite[Theorem 5.2.3, p.~101]{Ham82}.
\end{proof}

\begin{cor}
\label{c.invertible.psi}
Let $n\geq 2$, $1\leq q \leq n-1$, $a=1$, 
   $s$ and $k $ be positive integers,
   $0 < \lambda < \lambda' < 1$,
   $\delta \in (n/2,n)$, \textcolor{red}{$\delta \ne (n-1)$}.
Then the non-linear mapping
\begin{equation}
\label{eq.heat.pseudo.nl1}
 \! \varPhi_q \, d_q \, \varPsi_{\mu,q}  \mathcal{N}_q :
   \mathcal{F}^{k,\mathbf{s} (s,\lambda,\lambda',\delta)} 
_{T, \varLambda^{q}, {\mathcal D}_{d\oplus d^*}} \cap \mathcal{S}_{d^*}
 \to   \mathcal{F}^{k,\mathbf{s} (s,\lambda,\lambda',\delta)} 
_{T, \varLambda^{q}, {\mathcal D}_{d\oplus d^*}} \cap \mathcal{S}_{d^*}
\end{equation} 
is continuous and compact and the mapping 
\begin{equation}
\label{eq.heat.pseudo.nl2}
   I \! + \! \varPhi_q \, d_q \, \varPsi_ {\mu,q} \mathcal{N}_q :
   \mathcal{F}^{k,\mathbf{s} (s,\lambda,\lambda',\delta)} 
_{T, \varLambda^{q}, {\mathcal D}_{d\oplus d^*}} \cap \mathcal{S}_{d^*}
 \to   \mathcal{F}^{k,\mathbf{s} (s,\lambda,\lambda',\delta)} 
_{T, \varLambda^{q}, {\mathcal D}_{d\oplus d^*}} \cap \mathcal{S}_{d^*}
\end{equation}
is continuous, Fredholm, injective and open. 
\end{cor} 

\begin{proof} Taking into the account the continuity of the operator  \eqref{eq.Phiq.cont} 
for  $\delta \in (n/2,n)$, $\delta+1-n \not \in {\mathbb Z}_+$ we may  
argue as in the proof of Corollary \ref{c.invertible.psi.0}, replacing the scale 
$\mathcal{F}^{k,\mathbf{s} (s,\lambda,\lambda',\delta)} 
_{T, \varLambda^{q}}$ with the scale $\mathcal{F}^{k,\mathbf{s} (s,\lambda,\lambda',\delta)} 
_{T, \varLambda^{q}, {\mathcal D}_{d\oplus d^*}} \cap \mathcal{S}_{d^*}$ and formula 
\eqref{eq.cont.lin.0} with formulas \eqref{eq.cont.lin.1}, \eqref{eq.cont.lin.2}, 
to conclude that the 
non-linear operator \eqref{eq.heat.pseudo.nl1} 
is compact and continuous, too. 
Thus, as
in the proof of Corollary \ref{c.invertible.psi.0}, the statement 
follows now from the implicit function theorem 
for Banach spaces,   see  \cite[Theorem 5.2.3, p.~101]{Ham82}.
\end{proof}

Let us formulate the corresponding statement for equations \eqref{eq.NS}.

\begin{cor}
\label{c.open.NS.short}
Let $n\geq 2$, $1\leq q \leq n-1$, $a=1$,  
   $s$ and $k $ be positive integers,
   $0 < \lambda < \lambda' < 1$,
   $\delta \in (n/2,n)$, \textcolor{red}{$\delta \ne (n-1)$.} 
	Then, for any pair 
	$$
 (f^{(0)},u_0^{(0)}) 
\in \mathcal{F}^{k,\mathbf{s} (s,\lambda,\lambda',\delta)} _{T, \varLambda^q, {\mathcal D}
	_{d\oplus d^*}} 
 \times
   C^{2s+k+1,\lambda,\delta}_{\varLambda^q}  \cap \mathcal{S}_{d^\ast}
$$
admitting the solution $(u^{(0)},p^{(0)})$ 
to \eqref{eq.NS} in  
$\mathcal{F}^{k,\mathbf{s} (s,\lambda,\lambda',\delta)} _{T, \varLambda^q, {\mathcal D}
	_{d\oplus d^*}}  \times   \mathcal{F}^{k-1,\mathbf{s} (s-1,\lambda,\lambda',\delta-1)} 
	_{T, \varLambda^{q-1}, {\mathcal D}_{d\oplus d^*}}$, 
there is a number $\varepsilon > 0$ with the property that for all data 
$$
(f,u_0) 
\in \mathcal{F}^{k,\mathbf{s} (s,\lambda,\lambda',\delta)} _{T, \varLambda^q, {\mathcal D}
	_{d\oplus d^*}} 
 \times
   C^{2s+k+1,\lambda,\delta}_{\varLambda^q}  \cap \mathcal{S}_{d^\ast}
$$
satisfying the estimate
\begin{equation}
\label{eq.NS.open.est}
   \| f -   f^{(0)} 
   \|_{\mathcal{F}^{k,\mathbf{s} (s,\lambda,\lambda',\delta)} 
	_{T, \varLambda^q, {\mathcal D}
	_{d\oplus d^*}}}
 + \|u_0 -  u^{(0)} _0  \|_{C^{2s+k+1,\lambda,\delta} _{\varLambda^q}}
 < \varepsilon
\end{equation}
equations \eqref{eq.NS} have a unique solution in 
$   \mathcal{F}^{k,\mathbf{s} (s,\lambda,\lambda',\delta)} _{T, \varLambda^q, {\mathcal D}
	_{d\oplus d^*}} 
 \times
   \mathcal{F}^{k-1,\mathbf{s} (s-1,\lambda,\lambda',\delta-1)} _{T, \varLambda^{q-1}, 
	{\mathcal D}_{d\oplus d^*}} $.
\end{cor}

\begin{proof} Indeed, as we have seen 
Lemma \ref{l.heat.key1}  and the properties of the 
fundamental solution $\varPsi_\mu$ and the Leray-Helmholtz type 
projection $\varPhi_q d_q$ imply that the solution 
$(u^{(0)},p^{(0)})$ 
to \eqref{eq.NS} in  related to the data $(f^{(0)},u_0^{(0)})$ satisfies also 
the operator equation
$$
(I \! + \! \varPhi_q \, d_q \, \varPsi_{\mu,q} {\mathcal N}_q ) u^{(0)} = 
\varPhi_q \, d_q \, (\varPsi_{\mu,q} f^{(0)} + \varPsi_{\mu,q,0} u_0^{(0)})
$$
in the space $ \mathcal{F}^{k,\mathbf{s} (s,\lambda,\lambda',\delta)} _{T, \varLambda^q, {\mathcal D}
	_{d\oplus d^*}} $.

Estimate \eqref{eq.NS.open.est} and Corollary \ref{c.invertible.psi}
provide that the norm 
$$
\|\varPhi_q \, d_q (\varPsi_{\mu,q}  f +   \varPsi_{\mu,q,0}  u_0  - 
\varPsi_{\mu,q}  f^{(0)} 
 -  \varPsi_{\mu,q,0}  u^{(0)}_0)  \|_{
\mathcal{F}^{k,\mathbf{s} (s,\lambda,\lambda',\delta)} _{T, \varLambda^q, {\mathcal D}
	_{d\oplus d^*}} 
 \times
   \mathcal{F}^{k-1,\mathbf{s} (s-1,\lambda,\lambda',\delta-1)} _{T, \varLambda^{q-1}, 
	{\mathcal D}_{d\oplus d^*}} 
	}
$$ 
is sufficiently small for the operator equation 
\begin{equation} \label{eq.operator}
(I \! + \! \varPhi_q \, d_q \, \varPsi_{\mu,q} {\mathcal N}_q ) u = 
\varPhi_q \, d_q \, (\varPsi_{\mu,q} f + \varPsi_{\mu,q,0} u_0)
\end{equation}
to admit the unique solution in the space $ \mathcal{F}^{k,\mathbf{s} (s,\lambda,\lambda',\delta)} _{T, \varLambda^q, {\mathcal D}
	_{d\oplus d^*}} $. 

By the discussion in the proof of Theorem \ref{t.NS.deriv.unique}, see \eqref{eq.range.D} and Lemma \ref{l.deRham.Hoelder.t.new}, 
 \begin{equation} \label{eq.pseudo.1} 
\varPhi_q \, d_q \, \varPsi_{\mu,q} {\mathcal N}_q u = 
\varPhi_q \, d_q \, \varPsi_\mu M_1^{(q)} ((d_q \oplus d_{q-1}^* u,u),
\end{equation}
$$
d_q (I- \varPhi_q \, d_q)  \Big(\varPsi_{\mu,q}  f +  \varPsi_{\mu,q,0}  u_0 
- \varPsi_{\mu,q} {\mathcal N}_q u \Big) =0
$$
if 
$\delta \in (n/2,n)$, $\delta+1-n \not \in {\mathbb Z}_+$. Hence, applying statement 
(1) of Lemma \ref{l.deRham.Hoelder.t.new} we see that there is a unique form 
$\tilde p   \in
     \mathcal{F}^{k-1,\mathbf{s} (s,\lambda,\lambda',\delta-1)} _{T, \varLambda^{q-1}, 
	{\mathcal D}_{d\oplus d^*}} \cap {\mathcal S}_{d^*}$ satisfying 
 \begin{equation} \label{eq.pseudo.2} 
	d_{q-1} \tilde p = (I- \varPhi_q \, d_q)  \Big(\varPsi_{\mu,q}  f +  \varPsi_{\mu,q,0}  u_0 
- \varPsi_\mu  {\mathcal N}_q u\Big). 
\end{equation}
Taking in account \eqref{eq.pseudo.lin},
\eqref{eq.pseudo.lin.1}, \eqref{eq.pseudo.lin.2}, we conclude that 
\begin{equation*}  
u \! + \!  \varPsi_{\mu,q}  {\mathcal N}_q u + d_{q-1}\tilde  p  =
\varPsi_{\mu,q}  f +  \varPsi_{\mu,q,0}  u_0
\end{equation*}
Then the form $p = H_\mu \tilde  p$ belongs to the space 
$\mathcal{F}^{k-1,\mathbf{s} (s-1,\lambda,\lambda',\delta-1)} _{T, \varLambda^{q-1}, 
	{\mathcal D}_{d\oplus d^*}} \cap {\mathcal S}_{d^*}$. 
Again, the properties 
of the fundamental solution $\varPsi_\mu$ mean that the pair $(u,p) $ is a solution 
to the Cauchy problem \eqref{eq.NS}. Moreover for any solution $(u',p')$
to \eqref{eq.NS} in the  class $ \mathcal{F}^{k,\mathbf{s} (s,\lambda,\lambda',\delta)} _{T, \varLambda^q, {\mathcal D}
	_{d\oplus d^*}} $ is a solution to 
\begin{equation}  \label{eq.Cauchy.pseudo}
\left\{
\begin{array}{llll}
H_\mu u' \! + \!  \varPhi_q d_q 
 {\mathcal N}_q u'    & = & \varPhi_q d_q 
f  & (x,t) \in {\mathbb R}^n \times (0,T), 
\\
   u
& =
& u_0,
& (x,t) \in \mathbb{R}^n \times \{ 0 \}.
\end{array}
\right.
\end{equation}
As the solutions to \eqref{eq.Cauchy.pseudo} and \eqref{eq.operator} 
in the  class $ \mathcal{F}^{k,\mathbf{s} (s,\lambda,\lambda',\delta)} 
_{T, \varLambda^q, {\mathcal D}
	_{d\oplus d^*}} $ coincide, Corollary  \ref{c.invertible.psi} 
	yields that problem \eqref{eq.Cauchy.pseudo} has no more than one solution 
in the weighted H\"older spaces, i.e. $u'=u$.  Then the uniqueness of the solution  
$(u,p)$ to \eqref{eq.NS} follows from Lemma \ref{l.deRham.Hoelder.t.new} because
	$d_{q-1}  p = (I- \varPhi_q \, d_q)  \Big(  f  -  {\mathcal N}_q u\Big)$,
that was to be proved.
\end{proof}

Finally, in a similar way we obtain the statement corresponding to $a=0$.

\begin{cor}
\label{c.open.NS.short.0}
Let $n\geq 2$, $0\leq q \leq n$, $a=0$,  
   $s$ and $k $ be positive integers,
   $0 < \lambda < \lambda' < 1$,
   $\delta \in (n/2,+\infty)$. 
	Then, for any pair 
	$
(f^{(0)},u_0^{(0)}) 
\in \mathcal{F}^{k,\mathbf{s} (s,\lambda,\lambda',\delta)} _{T, \varLambda^q} 
 \times
   C^{2s+k+1,\lambda,\delta}_{\varLambda^q} 
$ 
admitting the solution $	u^{(0)}$ 
to \eqref{eq.NS} in the space 
$\mathcal{F}^{k,\mathbf{s} (s,\lambda,\lambda',\delta)} _{T, \varLambda^q}  $, 
there is a number $\varepsilon > 0$ with the property that for all data 
$
(f,u_0) 
\in \mathcal{F}^{k,\mathbf{s} (s,\lambda,\lambda',\delta)} _{T, \varLambda^q} 
 \times    C^{2s+k+1,\lambda,\delta}_{\varLambda^q}  
$ 
satisfying the estimate
\begin{equation}
\label{eq.NS.open.est.0}
   \| f -   f^{(0)} 
   \|_{\mathcal{F}^{k,\mathbf{s} (s,\lambda,\lambda',\delta)} 
	_{T, \varLambda^q}}
 + \|u_0 -  u^{(0)} _0  \|_{C^{2s+k+1,\lambda,\delta} _{\varLambda^q}}
 < \varepsilon
\end{equation}
equations \eqref{eq.NS} have a unique solution 
$u\in   \mathcal{F}^{k,\mathbf{s} (s,\lambda,\lambda',\delta)} _{T, \varLambda^q}  $.
\end{cor}

Thus, we see that there is crucial difference between problem 
\eqref{eq.NS} in the "local"{} situation where $a=0$ and the "non-local"{}
situation where $a=1$. As in the second case the problem is equivalent 
to a "pseudo-differential" Cauchy problem \eqref{eq.Cauchy.pseudo}, 
we observe some restrictions on possible asymptotic behaviour 
of solutions at the infinity with respect to the space variables and 
some additional loss of smoothness of the solutions. 
The reason is that we deal with scales of parabolic H\"{o}lder spaces, where 
the dilation principle is partially neglected with regard to the weight because
we need to provide some continuity of the integral operators $\varPhi_q$ 
and $\varPhi_q d_q$.

\bigskip

The investigation was supported by a grant of the Foundation for the advancement of theoretical
physics and mathematics ``BASIS''.

\end{document}